\documentclass[12pt,reqno]{amsart}

\usepackage{amsbsy}
\usepackage{amscd}
\usepackage{amsfonts}
\usepackage{amsmath}
\usepackage{amssymb}
\usepackage{amsthm}
\usepackage{amsxtra}

\usepackage[utf8]{inputenc}
\usepackage[margin=1in]{geometry}

\usepackage{datetime}
\usepackage{tikz-cd}
\usepackage{tikz}
\usepackage{todonotes}
\usepackage{verbatim}

\theoremstyle{definition}

\newtheorem{thm}{Theorem}[section]
\newtheorem*{thm*}{Theorem}

\newtheorem*{conj*}{Conjecture}

\newtheorem*{conv*}{Convention}

\newtheorem*{cor*}{Corollary}

\newtheorem*{defn*}{Definition}

\newtheorem*{exa*}{Example}
\newtheorem{exc}[thm]{Exercise}
\newtheorem*{exc*}{Exercise}

\newtheorem*{fact*}{Fact}
\newtheorem{lem}[thm]{Lemma}
\newtheorem*{lem*}{Lemma}

\newtheorem*{prob*}{Problem}

\newtheorem*{prop*}{Proposition}

\newtheorem*{ques*}{Question}
\newtheorem{rmk}[thm]{Remark}
\newtheorem*{rmk*}{Remark}

\newcommand{\C}{\mathbb{C}}

\newcommand{\Q}{\mathbb{Q}}

\newcommand{\R}{\mathbb{R}}

\newcommand{\Z}{\mathbb{Z}}

\newcommand{\cM}{\mathcal{M}}

\newcommand{\cP}{\mathcal{P}}

\newcommand{\al}{\alpha}
\newcommand{\bet}{\beta}
\newcommand{\Gam}{\Gamma}
\newcommand{\gam}{\gamma}
\newcommand{\Del}{\Delta}
\newcommand{\del}{\delta}
\newcommand{\eps}{\varepsilon}

\newcommand{\Lam}{\Lambda}

\newcommand{\Om}{\Omega}
\newcommand{\om}{\omega}

\newcommand{\ra}{\rightarrow}

\newcommand{\ol}{\overline}
\newcommand{\pr}{\prime}

\newcommand{\wt}{\widetilde}

\newcommand{\sm}{\setminus}

\DeclareMathOperator{\GL}{GL}
\DeclareMathOperator{\SL}{SL}

\DeclareMathOperator{\lcm}{lcm}

\DeclareMathOperator{\Per}{Per}

\DeclareMathOperator{\Pres}{Pres}

\DeclareMathOperator{\rank}{rank}
\DeclareMathOperator{\re}{Re}

\DeclareMathOperator{\Tw}{Tw}

\newcommand{\be}{\begin{equation*}}
\newcommand{\ee}{\end{equation*}}
\newcommand{\bex}{\begin{exc}}
\newcommand{\eex}{\end{exc}}
\newcommand{\bpf}{\begin{proof}}
\newcommand{\epf}{\end{proof}}

\title{Saturated orbit closures in the Hodge bundle}
\author[Karl Winsor]{Karl Winsor \\ \\ \monthname[\the\month] \the\day, \the\year}

\begin{document}

\maketitle

\begin{abstract}
We give a new proof of the classification of $\GL^+(2,\R)$-orbit closures that are saturated for the absolute period foliation of the Hodge bundle. As a consequence, we obtain a short proof of the classification of closures of leaves of the absolute period foliation of the Hodge bundle. Our approach is based on a method for classifying $\GL^+(2,\R)$-orbit closures using deformations of flat pairs of pants.
\end{abstract}


\section{Introduction}

The Hodge bundle $\Om\cM_g$ is the bundle of nonzero holomorphic $1$-forms on closed Riemann surfaces of genus $g$. The purpose of this paper is to develop the interplay between two well-known dynamical systems on the Hodge bundle.

For the first, a nonzero holomorphic $1$-form $(X,\om) \in \Om\cM_g$ induces an atlas of charts from the complement of the set of zeros of $\om$ to the complex plane $\C$, whose transition maps are translations in $\C$. The standard action of $\GL^+(2,\R)$ on $\C = \R + \R i$ induces an action on $\Om\cM_g$ by composing with these charts. For the second, integrating $\om$ over closed loops in $X$ gives a homomorphism $H_1(X;\Z) \ra \C$. The image $\Per(\om)$ of this homomorphism is the group of {\em absolute periods} of $\om$. The complex dimension of $\Om\cM_g$ is $4g-3$, provided $g \geq 2$. One can vary $(X,\om)$ while keeping $\Per(\om)$ constant to sweep out a leaf of a holomorphic foliation of $\Om\cM_g$ called the {\em absolute period foliation}. Leaves of the absolute period foliation have complex dimension $2g-3$.

The dynamics of $\GL^+(2,\R)$ and the absolute period foliation exhibit rich and complicated behavior. Both are ergodic with respect to the Lebesgue measure class on $\Om\cM_g$ \cite{Mas:IET}, \cite{Vee:Gauss}, \cite{Vee:flow}, \cite{McM:isoperiodic}, \cite{Ham:ergodicity}, \cite{CDF:transfer}. At the same time, their orbit closures (respectively, leaf closures) are sufficiently well-behaved to give hope for a complete classification. In particular, they are suborbifolds with simple descriptions in suitable local coordinates \cite{EMM:closures}, \cite{CDF:transfer}.

In this paper, we will give a short proof of the classification of closures of leaves of the absolute period foliation when $g \geq 3$, using the structure theory of $\GL^+(2,\R)$-orbit closures developed in \cite{EMM:closures}, \cite{AEM:symplectic}, and \cite{Wri:cylinder}. We first recall this classification, which was first proven in \cite{CDF:transfer}. Let $L$ be a leaf of the absolute period foliation of $\Om\cM_g$. Any two holomorphic $1$-forms in $L$ have the same area and the same group of absolute periods. Let $a > 0$ be this area, and let $\Lam \subset \C$ be the closure of this group of absolute periods. By the Riemann bilinear relations, $\Lam$ must contain a lattice in $\C$. There are three possibilities for $\Lam$.
\begin{enumerate}
    \item If the absolute periods are discrete, then $\Lam$ is a lattice in $\C$.
    \item If the absolute periods are neither discrete nor dense, then $\Lam = \ell + \Z w$ for some line $\ell$ through $0$ and some $w \notin \ell$.
    \item If the absolute periods are dense, then $\Lam = \C$.
\end{enumerate}
In the first case, every $(X,\om) \in L$ admits a branched covering to the torus $\C/\Lam$ defined by integrating $\om$, and $L$ is closed in $\Om\cM_g$. In the intermediate case $\Lam = \ell + \Z w$, define $\Om^\Lam\cM_g$ to be the set of holomorphic $1$-forms $(X,\om) \in \Om\cM_g$ such that $\ell + \Per(\om) = \Lam$. In other words, $\Per(\om)$ is contained in $\Lam$ and intersects every connected component of $\Lam$. Define $\Om_a\cM_g$ to be the set of holomorphic $1$-forms with area $a$, and define $\Om_a^\Lam\cM_g$ similarly. The classification of leaf closures for $g \geq 3$ from \cite{CDF:transfer} is as follows.

\begin{thm} \label{thm:closures}
Suppose $g \geq 3$. If $\Lam$ is not a lattice in $\C$, then the closure of $L$ is $\Om_a^\Lam\cM_g$. Otherwise, $L$ is closed.
\end{thm}

A subset of $\Om\cM_g$ is {\em saturated} for the absolute period foliation if it is a union of leaves of the absolute period foliation. One interesting consequence of Theorem \ref{thm:closures} is a classification of saturated orbit closures for the action of $\GL^+(2,\R)$ on the Hodge bundle.

\begin{thm} \label{thm:sat}
Suppose $g \geq 3$, and let $\cM$ be a saturated orbit closure in $\Om\cM_g$. If there is $(X,\om) \in \cM$ such that $\Per(\om)$ is not discrete, then $\cM = \Om\cM_g$. Otherwise, $\cM$ is a locus of torus covers.
\end{thm}

We now outline our approach to Theorem \ref{thm:closures}, which is very different from the one in \cite{CDF:transfer}. Our first observation is that Theorem \ref{thm:closures} can be quickly deduced from Theorem \ref{thm:sat}, using an inductive argument and a carefully chosen connected sum construction. We present this argument in Section \ref{sec:leaf}.

In Section \ref{sec:orbit}, we recall the results on orbit closures for the action of $\GL^+(2,\R)$ from \cite{EMM:closures}, \cite{AEM:symplectic}, and \cite{Wri:cylinder} that we will use in Section \ref{sec:sat} to prove Theorem \ref{thm:sat}. The starting point for our approach to Theorem \ref{thm:sat} is the fact that every orbit closure $\cM$ contains a horizontally periodic holomorphic $1$-form \cite{SW:minimal}. Using explicit local surgeries for navigating leaves of the absolute period foliation, we produce a holomorphic $1$-form $(X,\om)$ whose horizontal cylinders provide a flat pair of pants decomposition of $(X,\om)$. In this case, to show $\cM$ is dense in $\Om\cM_g$, results from \cite{Wri:cylinder} tell us it is enough to show that every horizontal cylinder on $(X,\om)$ can be sheared while remaining in $\cM$. Many linear combinations of these cylinder shears define paths inside the leaf of the absolute period foliation through $(X,\om)$. Using these shears and the structure of the tangent space to $\cM$ at a point in the principal stratum as a subspace of $H^1(X,Z(\om);\C)$, we find a {\em free pair of pants} on $(X,\om)$, that is, a flat pair of pants whose legs can be sheared and stretched freely while remaining in $\cM$. By stretching the legs of this pair of pants and navigating the associated leaves of the absolute period foliation, we are able to find enough free pairs of pants to show that every horizontal cylinder on $(X,\om)$ can be sheared while remaining in $\cM$. This last step is delicate, and requires travelling a large distance in a leaf of the absolute period foliation, causing many horizontal cylinders to collapse and re-emerge in the process, while still maintaining enough control on the horizontal cylinder decompositions to extract precise information about the tangent spaces to $\cM$. In general, the behavior of cylinder decompositions along leaves of the absolute period foliation can be quite wild, see for instance \cite{HW:Rel}, \cite{McM:isoperiodic}.

The methods introduced in this paper for proving Theorem \ref{thm:sat} are more broadly applicable to the problem of classifying orbit closures for the action of $\GL^+(2,\R)$ on {\em strata} $\Om\cM_g(k_1,\dots,k_n)$ of holomorphic $1$-forms with prescribed zero orders. In particular, much of the proof of Theorem \ref{thm:sat} only makes use of deformations supported near a single simple zero, and the proof can be adapted to obtain the following.

\begin{thm} \label{thm:free}
If $\cM$ is an orbit closure in a stratum such that there is $(X,\om) \in \cM$ with a free pair of pants, then $\cM$ is a connected component of a stratum.
\end{thm}

Theorem \ref{thm:free} admits many variations beyond strata with a simple zero, and we indicate some of these at the end of Section \ref{sec:sat}. The methods in our proof of Theorem \ref{thm:closures} can also be applied to the problem of classifying closures of leaves of the absolute period foliation of a stratum, using much more complicated connected sum constructions. We pursue these applications in forthcoming work. \\

\paragraph{\bf Notes and references.} The approach in \cite{CDF:transfer} for classifying closures of leaves of the absolute period foliation of $\Om\cM_g$ is very different and primarily topological. The classification theorem in \cite{CDF:transfer} is a consequence of a nearly complete answer to the question of when two homologically marked $1$-forms can be connected by a path of isoperiodic $1$-forms. Our approach and the approach in \cite{CDF:transfer} both build on results in \cite{McM:isoperiodic}, which established the ergodicity of the absolute period foliation on $\Om_1\cM_g$ for $g = 2,3$. Ergodicity on $\Om_1\cM_g$ for $g \geq 4$ was proven independently in \cite{CDF:transfer} and \cite{Ham:ergodicity}, and was proven for connected strata in \cite{Win:ergodic}. The existence of dense real Rel flow orbits in the area-$1$ locus of non-minimal stratum components is proven in \cite{Win:dense}, and a conditional theorem on the ergodicity of these real Rel flows is proven in ongoing work of Chaika-Weiss, conditional on a measure-rigidity theorem for the action of $\SL(2,\R)$ on products of strata. Explicit examples of dense leaves of the absolute period foliation in the area-$1$ locus of certain stratum components were given \cite{HW:Rel} and \cite{Ygo:dense}, and explicit full measure sets of dense leaves in the area-$1$ locus of connected strata were given in \cite{Win:ergodic}. The dynamics of absolute period foliations in strata of meromorphic differentials have been studied recently in \cite{CD:2poles} and \cite{KLS:realnorm}. The problem of classifying saturated orbit closures in strata of holomorphic $1$-forms has also been solved (as part of more general results) in genus $2$ in \cite{McM:SL2R}, in hyperelliptic components of strata \cite{Api:hyperelliptic}, \cite{Api:rank1}, and in genus $3$ for certain strata \cite{AN:rank22}, \cite{MW:fullrank}, \cite{Ygo:nonarith}. \\

\paragraph{\bf Acknowledgements.} During the preparation of this paper, the author was supported by an NSF GRFP under grant DGE-1144152 and by a Simons Postdoctoral Fellowship at the Fields Institute.


\section{Absolute period leaf closures} \label{sec:leaf}

In this section, we prove Theorem \ref{thm:closures} assuming Theorem \ref{thm:sat}. We begin with a flat-geometric connected sum construction that will be crucial to our proof of Theorem \ref{thm:closures}. For background on the Hodge bundle, the action of $\GL^+(2,\R)$, and the absolute period foliation (also known as the isoperiodic foliation or the kernel foliation), we refer to the survey \cite{Zor:survey}.

The {\em principal stratum} $\Om\cM_g(1^{2g-2})$ is the open dense subset of $\Om\cM_g$ consisting of holomorphic $1$-forms with $2g-2$ distinct simple zeros. Using standard surgeries for splitting higher-order zeros into lower-order zeros (see Section 8 in \cite{EMZ:principal}), it is easy to show that every leaf of the absolute period foliation of $\Om\cM_g$ intersects the principal stratum. Thus, for the rest of this section, we will work entirely in the principal stratum. We define $\Om_a\cM_g(1^{2g-2})$ and $\Om_a^\Lam\cM_g(1^{2g-2})$ analogously as in the introduction. {\em Local period coordinates} on $\Om\cM_g(1^{2g-2})$ are defined by integrating holomorphic $1$-forms over a basis for the relative homology group $H_1(X,Z(\om);\Z)$, where $Z(\om)$ is the set of zeros. \\

\paragraph{\bf Star-shaped connected sums.} Fix $g \geq 3$. For $1 \leq j \leq g$, choose a lattice $\Lam_j = \Z a_j + \Z b_j$ in $\C$, and let $T_j = (\C/\Lam_j, dz)$ be the associated flat torus. For $2 \leq j \leq g$, choose embedded oriented segments $s_j \subset T_1$ and $s_j^\pr \subset T_j$ such that $\int_{s_j} dz = \int_{s_j^\pr} dz$, and let $c_j = \int_{s_j} dz$. We assume that the segments $s_2,\dots,s_g$ are disjoint. For $3 \leq j \leq g$, choose a path $\gam_j \subset T_1$ from the starting point of $s_2$ to the starting point of $s_j$, such that $\gam_j$ is disjoint from all of the interiors of the segments $s_k,s_k^\pr$, and let $d_j = \int_{\gam_j} dz$. For $2 \leq j \leq g$, slit $T_1$ along $s_j$, slit $T_j$ along $s_j^\pr$, and reglue opposite sides. The result is a holomorphic $1$-form $(X,\om) \in \Om\cM_g(1^{2g-2})$ with a pair of simple zeros arising from the endpoints of each pair of segments $s_k,s_k^\pr$. We will call a presentation of a holomorphic $1$-form as a connected sum of $g$ flat tori as above a {\em star-shaped connected sum}. See Figure \ref{fig:sscs}.

Star-shaped connected sums provide parameterizations of open subsets of $\Om\cM_g(1^{2g-2})$. Indeed, the parameters $a_j,b_j$ are integrals of $\om$ along closed geodesics $\al_j,\bet_j$ in $T_j$, and the parameters $c_j,d_j$ are integrals of $\om$ along paths in $T_1$ between zeros of $\om$. The associated relative homology classes form a basis of $H_1(X,Z(\om);\Z)$. Thus, the $4g-3$ complex numbers $\{a_j,b_j\}_{j=1}^g \cup \{c_j\}_{j=2}^g \cup \{d_j\}_{j=3}^g$ provide local period coordinates. The action of $\GL^+(2,\R)$ on $\Om\cM_g(1^{2g-2})$ respects star-shaped connected sums.

\begin{figure}
    \centering
    \includegraphics[width=0.6\textwidth]{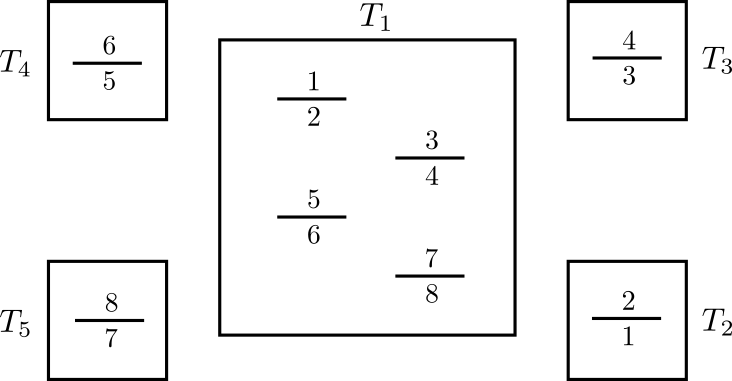}
    \caption{An example of a star-shaped connected sum in genus $5$.}
    \label{fig:sscs}
\end{figure}

\begin{lem} \label{lem:sum}
If $(X,\om) \in \Om\cM_g(1^{2g-2})$ is such that $\Per(\om)$ is not discrete, then there exists a star-shaped connected sum in the leaf of the absolute period foliation through $(X,\om)$.
\end{lem}

\begin{proof}
The set of holomorphic $1$-forms in $\Om\cM_g(1^{2g-2})$ that do not satisfy the conclusion of the lemma is a proper closed $\GL^+(2,\R)$-invariant subset that is saturated for the absolute period foliation. By Theorem \ref{thm:sat}, every holomorphic $1$-form in this set has a discrete group of absolute periods.
\end{proof}

\begin{lem} \label{lem:nondiscrete}
Let $\Lam_1,\dots,\Lam_g$ be lattices in $\C$ such that $\Lam_1 + \cdots + \Lam_g$ is not discrete. Then there is $2 \leq j \leq g$ such that $\Lam_1 + \Lam_j$ is not discrete. Moreover, if $\Lam_1 + \cdots + \Lam_g$ is dense, then there is $2 \leq k \leq g$ such that $\Lam_1 + \Lam_j + \Lam_k$ is dense.
\end{lem}

\begin{proof}
For the first claim, suppose that $\Lam_1 + \Lam_j$ is discrete for all $2 \leq j \leq g$. Since $\Lam_1$ is a lattice, the index $[\Lam_1 + \Lam_j : \Lam_1]$ is finite, so $\Lam_1 + \Lam_j \subset \frac{1}{n_j}\Lam_1$ for some positive integer $n_j$. Letting $N = \lcm_j n_j$, we see that $\Lam_1 + \cdots + \Lam_g \subset \frac{1}{N} \Lam_1$ is discrete, a contradiction. For the second claim, we know that some $\Lam_1 + \Lam_j$ is not discrete. Let $\Lam_1^\pr$ be the closure of $\Lam_1 + \Lam_j$. We may assume $\Lam_1^\pr \neq \C$. Then $\Lam_1^\pr = \ell + \Z w$ for some line $\ell$ through the origin and some $w \notin \ell$. Let $\Lam_k^\pr$ be the closure of $\Lam_1 + \Lam_j + \Lam_k$, and suppose that $\Lam_k^\pr \neq \C$ for all $2 \leq k \leq g$. Then $[\Lam_k^\pr : \Lam_1^\pr]$ is finite, so $\Lam_k^\pr \subset \frac{1}{n_k^\pr} \Lam_1^\pr$ for some positive integer $n_k^\pr$. Letting $N^\pr = \lcm_k n_k^\pr$, we see that $\Lam_1 + \cdots + \Lam_g \subset \frac{1}{N^\pr} \Lam_1^\pr$ is not dense, a contradiction.
\end{proof}

We are now ready to prove Theorem \ref{thm:closures}. Our only use of Theorem \ref{thm:sat} is in Lemma \ref{lem:sum} above, which allows us to restrict our attention to leaves of the absolute period foliation containing star-shaped connected sums.

\begin{proof}
(of Theorem \ref{thm:closures}) We induct on the genus $g$. The base cases $g = 2,3$ are a consequence of Theorem 1.1 in \cite{McM:isoperiodic}, as shown in Theorem 3.3 and Section 3.5 in \cite{CDF:transfer}.

Suppose $g \geq 4$. Fix $(X,\om) \in \Om\cM_g(1^{2g-2})$ such that $\Per(\om)$ is not discrete, and let $L$ be the leaf of the absolute period foliation through $(X,\om)$. By Lemma \ref{lem:sum}, we may assume that $(X,\om)$ is a star-shaped connected sum. Let $T_j = (\C/\Lam_j, dz)$, $1 \leq j \leq g$, be the associated flat tori, and let $s_k,s_k^\pr$, $2 \leq k \leq g$, be the associated slits. Let $a_j$ be the area of $T_j$, so $a = a_1 + \cdots + a_g$ is the area of $(X,\om)$. Throughout, we tacitly assume that the slits $s_k,s_k^\pr$ are very short, and that the distances between the slits on $T_1$ are very large relative to their lengths, since this can always be arranged by moving along $L$.

For $2 \leq j \leq g$, let $(X_j,\om_j) \in \Om\cM_{g-1}(1^{2g-4})$ be obtained from $(X,\om)$ by ``forgetting'' $T_j$ and the slits $s_j,s_j^\pr$. In other words, $(X_j,\om_j)$ is obtained by slitting $s_j,s_j^\pr$, regluing opposite sides to obtain a disconnected surface with $2$ components, and taking the component of genus $g-1$. Let $L_j$ be the leaf of the absolute period foliation of $\Om\cM_{g-1}(1^{2g-4})$ through $(X_j,\om_j)$. By Lemma \ref{lem:nondiscrete}, after permuting $T_2,\dots,T_g$, we may assume that $\Lam_1 + \Lam_2$ is not discrete. Then since $g \geq 4$, both $\Per(\om_g)$ and $\Per(\om_{g-1})$ are not discrete. We consider $2$ cases depending on the closure of $\Per(\om)$. \\

\paragraph{\bf Case 1.} Suppose $\Per(\om)$ is dense. By Lemma \ref{lem:nondiscrete}, at least one of $\Per(\om_g)$ or $\Per(\om_{g-1})$ is dense. After swapping $T_{g-1}$ and $T_g$, we may assume $\Per(\om_g)$ is dense. Then by induction, $L_g$ is dense in $\Om_{a-a_g}\cM_{g-1}(1^{2g-4})$. Choose lattices $\Lam_j^\pr$ for $1 \leq j \leq g - 1$ such that the areas of the flat tori $T_j^\pr = (\C/\Lam_j^\pr,dz)$ sum to $a-a_g$. Let $(Y,\eta) \in \Om_a\cM_g(1^{2g-2})$ be a star-shaped connected sum with tori $T_1^\pr,\dots,T_{g-1}^\pr,T_g$, and let $(Y_g,\eta_g) \in \Om_{a-a_g}\cM_{g-1}(1^{2g-4})$ be obtained from $(Y,\eta)$ by forgetting $T_g$ and the slits $s_g,s_g^\pr$. There is a sequence in $L_g$ approaching $(Y_g,\eta_g)$. By connecting consecutive points in this sequence along paths in $L_g$, we can produce a continuous map $\varphi_g : [0,\infty) \ra L_g$ such that the points $\varphi_g(n)$, $n \in \Z_{\geq 0}$, approach $(Y_g,\eta_g)$ as $n \ra \infty$. By compactness, there is $\eps_n > 0$ such that on any holomorphic $1$-form in $\varphi_g([0,n])$, the minimum distance between two distinct zeros is greater than $\eps_n$. After shrinking the slits $s_g,s_g^\pr$ on $(X,\om)$ to have length less than $\eps_n$, the path $\varphi_g([0,n])$ determines a path $\varphi([0,n])$ in $L$ by deforming $(X_g,\om_g)$ according to $\varphi_g([0,n])$ while keeping $T_g$ and the slits $s_g,s_g^\pr$ fixed. By restoring the slits $s_g,s_g^\pr$ on $\varphi(n)$ to their original length and direction, we obtain a sequence in $L$ approaching $(Y,\eta)$. In short, we may arbitrarily modify $(X_g,\om_g)$ while preserving its area $a - a_g$ (keeping the slits very short and very far apart relative to their length), while staying in the closure of $L$.

By rotating $\Lam_1$ slightly, we can arrange that $\Lam_1 + \Lam_j$ is dense for all $2 \leq j \leq g$. Since we are not changing the area of any of the $T_j$ in the process, the previous paragraph shows in particular that we can rotate $\Lam_1$ while staying in the closure of $L$. Thus, we may assume that $\Lam_1 + \Lam_j$ is dense for all $2 \leq j \leq g$, so that $L_j$ is dense in $\Om_{a-a_j}\cM_{g-1}(1^{2g-4})$ for all $2 \leq j \leq g$. By iterating the argument from the previous paragraph, we obtain the following. Let $\Del$ be the open simplex of tuples $(b_1,\dots,b_g)$ with $b_j > 0$ for all $j$ such that $\sum_{j=1}^g b_j = a$. Consider the equivalence relation $\sim$ on $\Del$ generated by $(b_1,\dots,b_g) \sim (b_1^\pr,\dots,b_g^\pr)$ whenever $b_j^\pr = b_j$ for some $2 \leq j \leq g$. For any $(b_1,\dots,b_g) \in \Del$ in the equivalence class of $(a_1,\dots,a_g)$, and for any star-shaped connected sum $(Y,\eta) \in \Om_a\cM_g(1^{2g-2})$ whose associated tori have areas $b_1,\dots,b_g$, the closure of $L$ contains $(Y,\eta)$.

Fix $(b_1,\dots,b_g) \in \Del$. Since $g \geq 4$, for $\eps > 0$ sufficiently small we have
\begin{align*}
(a_1,\dots,a_g) &\sim (\eps,\dots,\eps, a_1 + \cdots + a_{g-1} - (g-2)\eps, a_g) \sim (\eps,\dots,\eps, a - (g-1)\eps) \\
(b_1,\dots,b_g) &\sim (\eps,\dots,\eps, b_1 + \cdots + b_{g-1} - (g-2)\eps, b_g) \sim (\eps,\dots,\eps, a - (g-1)\eps)
\end{align*}
so $(a_1,\dots,a_g) \sim (b_1,\dots,b_g)$. Thus, every star-shaped connected sum in $\Om_a\cM_g(1^{2g-2})$ lies in the closure of $L$. Since the action of $\SL(2,\R)$ on $\Om_a\cM_g(1^{2g-2})$ is ergodic, the set of star-shaped connected sums in $\Om_a\cM_g(1^{2g-2})$ is dense, thus $L$ is dense in $\Om_a\cM_g(1^{2g-2})$. \\

\paragraph{\bf Case 2.} Suppose that $\Per(\om)$ is not dense. Since $\Per(\om)$ contains a lattice and is not discrete, the closure of $\Per(\om)$ is given by $\ell + \Z w$ for some line $\ell$ through $0$ and some $w \notin \ell$. By applying an element of $\SL(2,\R)$ to $(X,\om)$, we may assume that the closure of $\Per(\om)$ is $\R + i\Z$. Then $\R + \Per(\om) = \R + i\Z$, and for $1 \leq j \leq g$, we have $\R + \Lam_j = \R + im_j\Z$ for some positive integer $m_j$. Moreover, $\gcd_j m_j = 1$. The action of the matrix $u_t = \begin{pmatrix}1 & t \\ 0 & 1\end{pmatrix}$ on $\C$ preserves the subgroup $\R + i n \Z$, and the action of $u_t$ on $\Om\cM_g$ preserves $\Om_a^{\R + i n \Z}\cM_g$, for all $n > 0$.

Since $\Per(\om_g)$ is not discrete, the closure of $\Per(\om_g)$ is given by $\R + im\Z$ for some positive integer $m$. By induction, the closure of $L_g$ is $\Om_{a-a_g}^{\R+im\Z}\cM_{g-1}(1^{2g-4})$. Following the argument in the first paragraph of Case 1, we may arbitrarily modify $(X_g,\om_g)$ within $\Om_{a-a_g}^{\R+im\Z}\cM_{g-1}(1^{2g-4})$ (keeping the slits very short and very far apart relative to their length), while staying in the closure of $L$. By applying $u_t$ to $\Lam_1$ for some small $t \in \R$, we can arrange that $\Lam_1 + \Lam_j$ is not discrete for all $2 \leq j \leq g$. Then we can iteratively apply the argument from the first paragraph of Case 1, and the result is the following. Let $\cP$ be the set of tuples of positive integers $(m_1,\dots,m_g)$ such that $\gcd_j m_j = 1$. Consider the equivalence relation on $\Del \times \cP$ generated by
\be
((b_1,\dots,b_g),(n_1,\dots,n_g)) \sim ((b_1^\pr,\dots,b_g^\pr),(n_1^\pr,\dots,n_g^\pr))
\ee
whenever $b_j^\pr = b_j$, $n_j^\pr = n_j$, and $\gcd_{k \neq j} n_k^\pr = \gcd_{k \neq j} n_k$ for some $2 \leq j \leq g$. For any $((b_1,\dots,b_g),(n_1,\dots,n_g)) \in \Del \times \cP$ in the equivalence class of $((a_1,\dots,a_g),(m_1,\dots,m_g))$, and for any star-shaped connected sum $(Y,\eta) \in \Om_a^{\R+i\Z}\cM_g(1^{2g-2})$ whose associated tori have areas $b_1,\dots,b_g$ and periods $\Lam_j$ with $\R + \Lam_j = \R + in_j\Z$, the closure of $L$ contains $(Y,\eta)$.

By the last paragraph of Case 1, equivalence classes in $\Del \times \cP$ are unions of fibers of the projection $\Del \times \cP \ra \cP$. Therefore, it is enough to show that the induced equivalence relation $\sim_\cP$ on $\cP$ has only one equivalence class. Let $m_j^\pr = \gcd_{k \neq j} m_k$ for $2 \leq j \leq g$, and note that $\gcd(m_j^\pr,m_j) = 1$. Since $g \geq 4$, we have
\be
(m_1,\dots,m_g) \sim_\cP (m_g^\pr,\dots,m_g^\pr,m_g) \sim_\cP (1,\dots,1,m_g^\pr,1) \sim_\cP (1,\dots,1) .
\ee
Therefore, $\sim_\cP$ has only one equivalence class, thus $L$ is dense in $\Om_a^{\R+i\Z}\cM_g(1^{2g-2})$.
\end{proof}

\begin{rmk}
Our proof of Theorem \ref{thm:closures} crucially relies on the genus $2$ case as a base case. However, it is not necessary to rely on the genus $3$ case. In the genus $3$ case, one must address the possibility of $(X_3,\om_3)$ being a nonarithmetic eigenform for real multiplication (see \cite{McM:isoperiodic}). In that case, the closure of $L_3$ contains the $\SL(2,\R)$-orbit of $(X_3,\om_3)$, and by applying $\SL(2,\R)$ only to $(X_3,\om_3)$, one can arrange that $(X_2,\om_2)$ has a dense group of absolute periods and is not an eigenform. The rest of the argument is then similar to Case 1 of the proof of Theorem \ref{thm:closures}.
\end{rmk}


\section{Background on orbit closures} \label{sec:orbit}

Before proving Theorem \ref{thm:sat}, we recall some of the structure theory of orbit closures for the action of $\GL^+(2,\R)$ on strata of holomorphic $1$-forms. Throughout, $\cM$ denotes an orbit closure in a stratum, and $g$ is the genus of the underlying surfaces.

By Theorem 2.1 in \cite{EMM:closures}, $\cM$ is a properly immersed affine suborbifold of its stratum, and is defined in local period coordinates by homogeneous linear equations with real coefficients. In general, $\cM$ may have self-intersection points, and by Theorem 2.2 in \cite{EMM:closures}, the self-intersection points form a finite union of smaller orbit closures.

The {\em field of definition} $k(\cM)$ is the smallest subfield $K \subset \R$ such that $\cM$ is defined in local period coordinates by homogeneous linear equations with coefficients in $K$. By Theorem 1.1 in \cite{Wri:field}, $k(\cM)$ is a number field of degree at most $g$. By Theorem 1.9 in \cite{Wri:cylinder}, if $k(\cM) \neq \Q$ and $(X,\om) \in \cM$ has at least one horizontal cylinder, then there exist two horizontal cylinders on $(X,\om)$ whose circumferences have an irrational ratio.

Away from the self-intersection points of $\cM$, the tangent space $T_{(X,\om)}\cM$ is naturally a complex subspace of the relative cohomology group $H^1(X,Z(\om);\C)$. The {\em real tangent space} is defined by
\be
T^\R_{(X,\om)} \cM = T_{(X,\om)}\cM \cap H^1(X,Z(\om);\R) .
\ee
Let $p$ be the projection from relative cohomology to absolute cohomology. By Theorem 1.4 in \cite{AEM:symplectic}, $p(T^\R_{(X,\om)}\cM)$ is a symplectic subspace of $H^1(X;\R)$, thus $p(T_{(X,\om)}\cM)$ has even complex dimension. The {\em rank} of $\cM$ is defined by
\be
\rank(\cM) = \frac{1}{2} \dim_\C p(T_{(X,\om)}\cM)
\ee
and is an integer satisfying $1 \leq \rank(\cM) \leq g$.

A holomorphic $1$-form is {\em horizontally periodic} if the underlying surface is a union of horizontal cylinders and horizontal saddle connections. By Corollary 6 in \cite{SW:minimal}, $\cM$ contains a horizontally periodic holomorphic $1$-form. For $(X,\om) \in \cM$ horizontally periodic, the {\em cylinder preserving space}
\be
\Pres_{(X,\om)}\cM \subset T^\R_{(X,\om)}\cM
\ee
is the real subspace of relative cohomology classes that evaluate to zero on every horizontal closed geodesic in $X \sm Z(\om)$, and the {\em twist space}
\be
\Tw_{(X,\om)}\cM \subset \Pres_{(X,\om)}\cM
\ee
is the real subspace of relative cohomology classes that evaluate to zero on every horizontal saddle connection in $X$. By Lemma 8.6 in \cite{Wri:cylinder}, if $\Tw_{(X,\om)}\cM$ is not equal to $\Pres_{(X,\om)}\cM$, then there is a horizontally periodic holomorphic $1$-form in $\cM$ with more horizontal cylinders than $(X,\om)$. By Lemma 8.10 in \cite{Wri:cylinder}, $p(\Tw_{(X,\om)}\cM)$ is an isotropic subspace of $H^1(X;\R)$, thus
\be
\rank(\cM) \geq \dim_\R p(\Tw_{(X,\om)}\cM) .
\ee
Lastly, by Corollary 8.11 in \cite{Wri:cylinder}, $\rank(\cM)$ is bounded above by the codimension of $\Tw_{(X,\om)}\cM$ in $T^\R_{(X,\om)}\cM$.


\section{Saturated orbit closures} \label{sec:sat}

We now prove Theorem \ref{thm:sat}. Since the principal stratum $\Om\cM_g(1^{2g-2})$ is open and dense in $\Om\cM_g$, and since every leaf of the absolute period foliation of $\Om\cM_g$ intersects the principal stratum, we may restrict our attention to the action of $\GL^+(2,\R)$ on the principal stratum. Throughout this section, unless explicitly stated otherwise, $g \geq 3$ is fixed and $\cM$ is an orbit closure for the action of $\GL^+(2,\R)$ on the principal stratum that is saturated for the absolute period foliation of the principal stratum.

Our approach to Theorem \ref{thm:sat} centers around deformations of flat pairs of pants consisting of horizontal cylinders. We begin by explaining how these pairs of pants arise. Suppose $(X,\om) \in \cM$ is horizontally periodic, and fix $Z \in Z(\om)$. The cone angle around $Z$ is $4\pi$, so there are exactly $2$ horizontal saddle connections $\gam_1^Z,\gam_2^Z$ emanating rightward from $Z$. Suppose that $\gam_1^Z$ and $\gam_2^Z$ are loops. Let $\gam_{j,\ell}^Z,\gam_{j,r}^Z$ be the left and rights ends of $\gam_j^Z$, respectively. Rotating counterclockwise through an angle of $\pi$ around $Z$ determines a permutation on the $4$ ends $\gam_{1,\ell}^Z,\gam_{1,r}^Z,\gam_{2,\ell}^Z,\gam_{2,r}^Z$. This permutation is a $4$-cycle that exchanges the subsets $\{\gam_{1,\ell}^Z,\gam_{2,\ell}^Z\}$ and $\{\gam_{1,r}^Z,\gam_{2,r}^Z\}$. There are only $2$ possible permutations. In the first case, the counterclockwise angle from $\gam_{1,r}^Z$ to $\gam_{1,\ell}^Z$ is $\pi$. Each $\gam_j^Z$ is the unique horizontal saddle connection in the top boundary of a horizontal cylinder $C_j^Z$, and $\gam_1^Z \cup \gam_2^Z$ is the bottom boundary of a horizontal cylinder $C_0^Z$. In the second case, the counterclockwise angle from $\gam_{1,\ell}^Z$ to $\gam_{1,r}^Z$ is $\pi$. Each $\gam_j^Z$ is the unique horizontal saddle connection in the bottom boundary of a horizontal cylinder $C_j^Z$, and $\gam_1^Z \cup \gam_2^Z$ is the top boundary of a horizontal cylinder $C_0^Z$. In either case, $Z$ is only contained in one boundary component of each of the cylinders $C_j^Z$. Choosing closed geodesics $\al_j^Z \subset C_j^Z$, the closure of the component of $X \sm (\al_0^Z \cup \al_1^Z \cup \al_2^Z)$ containing $Z$ is a pair of pants. See Figure \ref{fig:fpp}.

\begin{figure}
    \centering
    \includegraphics[width=0.8\textwidth]{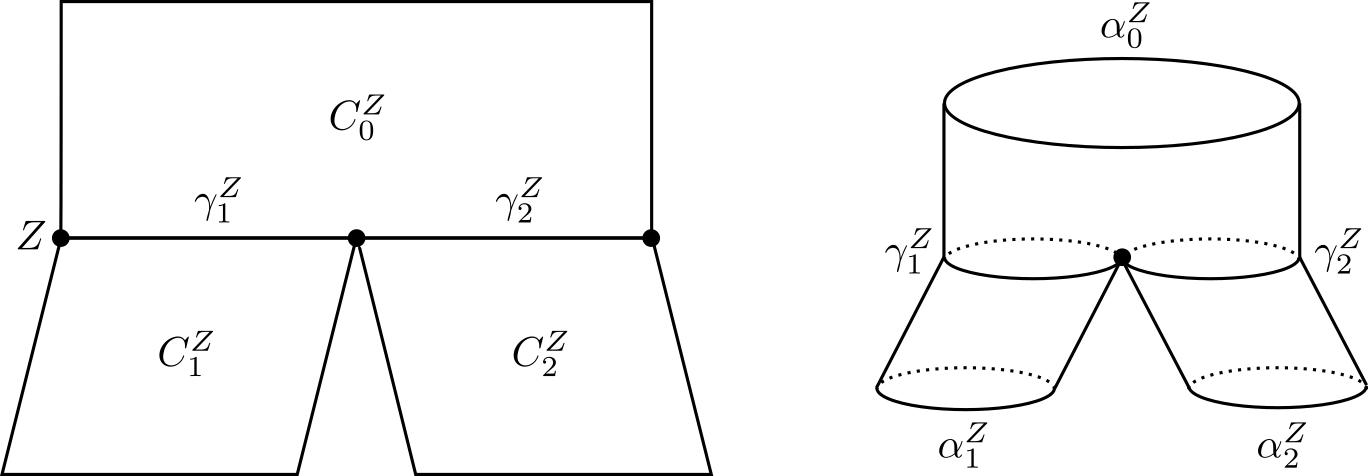}
    \caption{An example of a flat pair of pants.}
    \label{fig:fpp}
\end{figure}

\begin{lem} \label{lem:pop}
There is a horizontally periodic $(X,\om) \in \cM$ with $3g-3$ horizontal cylinders. Every horizontal saddle connection on $(X,\om)$ is a loop.
\end{lem}

\begin{proof}
By Corollary 6 in \cite{SW:minimal}, there is a horizontally periodic $(X_0,\om_0) \in \cM$. Suppose there is a horizontal saddle connection $\gam$ on $(X_0,\om_0)$ with distinct endpoints, and let $Z$ be one of the endpoints of $\gam$. Recall that the cone angle around $Z$ is $4\pi$. Fix $\eps > 0$ smaller than the minimum height of a horizontal cylinder on $(X_0,\om_0)$, and let $s_1,s_2$ be the two upward segments of length $\eps$ emanating from $Z$. Let $C_j$ be the horizontal cylinder containing the interior of $s_j$, and let $w_j$ be the circumference of $C_j$. By slitting $(X_0,\om_0)$ along $s_1$ and $s_2$ and regluing opposite sides, we obtain a horizontally periodic holomorphic $1$-form $(X_1,\om_1)$ on which both horizontal saddle connections through $Z$ are loops. The heights of $C_1$ and $C_2$ decrease by $\eps$. If $C_1 \neq C_2$, one new horizontal cylinder of height $\eps$ and circumference $w_1 + w_2$ is formed, and if $C_1 = C_2$, two new horizontal cylinders of height $\eps$ are formed and their circumferences sum to $w_1$. This surgery does not change the absolute periods of $\om$. See Figure \ref{fig:rel}.

Every horizontal saddle connection disjoint from $Z$ is preserved in the process. Thus, we have increased the number of zeros $Z_0$ with the property that both horizontal saddle connections through $Z_0$ are loops. By iterating this procedure, we can perturb $(X_0,\om_0)$ along the associated leaf of the absolute period foliation to get a horizontally periodic holomorphic $1$-form $(X,\om)$ on which every horizontal saddle connection is a loop. Since $\cM$ is saturated for the absolute period foliation, $(X,\om) \in \cM$.

By the discussion preceding the lemma, each zero of $\om$ lies in the boundary of exactly $3$ distinct horizontal cylinders on $(X,\om)$, and each horizontal cylinder contains exactly $2$ distinct zeros in its boundary (one zero in the top boundary and one zero in the bottom boundary). Since there are $2g-2$ zeros, there are $3g-3$ horizontal cylinders.
\end{proof}

\begin{figure}
    \centering
    \includegraphics[width=0.4\textwidth]{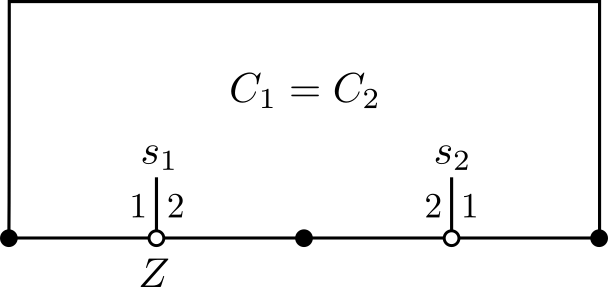}
    \caption{An illustration of the surgery in the proof of Lemma \ref{lem:pop} in the case $C_1 = C_2$.}
    \label{fig:rel}
\end{figure}

We fix some notation that we will use for the rest of the section. Fix a horizontally periodic $(X,\om) \in \cM$ with $3g-3$ horizontal cylinders. For each zero $Z \in Z(\om)$, we denote by $\gam_1^Z,\gam_2^Z$ the $2$ horizontal saddle connections through $Z$, and we denote by $C_0^Z,C_1^Z,C_2^Z$ the $3$ horizontal cylinders containing $Z$ in their boundary, chosen so that the circumferences $w_0^Z,w_1^Z,w_2^Z$ satisfy
\be
w_0^Z = w_1^Z + w_2^Z
\ee
and so that $\gam_1^Z$ is a boundary component of $C_1^Z$ and $\gam_2^Z$ is a boundary component of $C_2^Z$. Then $\gam_1^Z \cup \gam_2^Z$ is a boundary component of $C_0^Z$. Here, the choice of $C_1^Z,C_2^Z$ is arbitrary. We denote the heights of these cylinders by $h_0^Z,h_1^Z,h_2^Z$.

Since $\cM$ is defined by homogeneous linear equations in local period coordinates, a tangent vector $\eta \in T_{(X,\om)}\cM \subset H^1(X,Z(\om);\C)$ determines a path $(X,\om) + t\eta$ in $\cM$, which in general is only defined for $t \in \R$ sufficiently small. For a horizontal cylinder $C$ on $(X,\om)$, the {\em cylinder shear} $\eta_C \in \Tw_{(X,\om)}\cM$ is the relative cohomology class that evaluates to the height of $C$ on any saddle connection $\gam \subset C \cup Z(\om)$ that crosses $C$ from bottom to top, and evaluates to $0$ on any path disjoint from $C$ and on any closed loop contained in $C$. In this case, the path $(X,\om) + t\eta_C$ is well-defined for all $t \in \R$ and projects to a loop in $\cM$ with period $w/h$, where $h$ and $w$ are the height and circumference of $C$. On $(X,\om) + t\eta_C$, the integral along $\gam$ is given by $\int_\gam \om + th$. The {\em standard shear} $\eta_\om \in T_{(X,\om)}\cM$ is the sum of all of the cylinder shears $\eta_C$ and is tangent to the horocycle through $(X,\om)$.

The shear $\eta_Z = \sum_{j=0}^2 \frac{1}{h_j^Z} \eta_{C_j^Z}$ lies in $\ker(p) \cap \Tw_{(X,\om)}\cM$, and the associated path $(X,\om) + t \eta_Z$ is contained in the leaf of the absolute period foliation through $(X,\om)$. Along this path, $C_0^Z$ is twisted by some amount while $C_1^Z$ and $C_2^Z$ are twisted by that same amount in the opposite direction. See Figure \ref{fig:shear}. This path can be viewed as an orbit of the linear flow $v \mapsto v + (t,-t,-t)$ in the torus $\prod_{j=0}^2 \R/w_j^Z\Z$. We then have the following standard lemma (see Section 3 in \cite{Wri:cylinder}).

\begin{lem} \label{lem:reltwist}
Fix a horizontally periodic $(X,\om) \in \cM$ with $3g-3$ horizontal cylinders, and fix $Z \in Z(\om)$. Let $1 \leq d \leq 3$ be the $\Q$-dimension of the $\Q$-span of the reciprocals $\frac{1}{w_0^Z},\frac{1}{w_1^Z},\frac{1}{w_2^Z}$. Then $\Tw_{(X,\om)}\cM \cap {\rm span}_\R (\eta_{C_0^Z},\eta_{C_1^Z},\eta_{C_2^Z})$ has dimension at least $d$.
\end{lem}

If $\cM$ has rank $1$ and field of definition $\Q$, then for any $(X,\om) \in \cM$, since $p(T_{(X,\om)}\cM)$ has dimension $2$, any absolute period of $\om$ can be expressed as a $\Q$-linear combination of $2$ fixed absolute periods. Then $\Per(\om)$ is a lattice in $\C$, and the map
\be
X \ra \C / \Per(\om), \quad x \mapsto \int_{x_0}^x \om
\ee
exhibits $X$ as a branched cover of a torus. Thus, $\cM$ is a locus of torus covers. The remaining case in rank $1$ is easy to rule out, as follows.

\begin{figure}
    \centering
    \includegraphics[width=\textwidth]{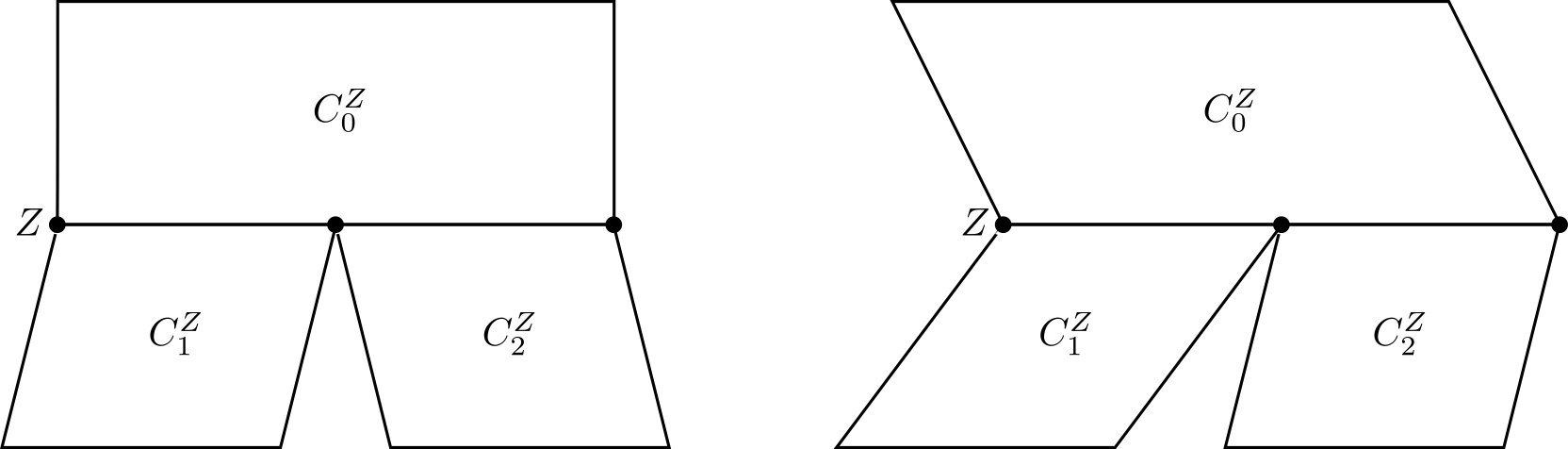}
    \caption{An illustration of the path $(X,\om) + t\eta_Z$, showing only the horizontal cylinders being sheared.}
    \label{fig:shear}
\end{figure}

\begin{lem} \label{lem:rank2}
If $k(\cM) \neq \Q$, then $\rank(\cM) \geq 2$.
\end{lem}

\begin{proof}
By Lemma \ref{lem:pop}, there is a horizontally periodic $(X,\om) \in \cM$ with $3g-3$ horizontal cylinders. By Theorem 1.9 in \cite{Wri:cylinder}, there are two horizontal cylinders on $(X,\om)$ whose circumferences have an irrational ratio. It follows that there are two such horizontal cylinders that share a horizontal saddle connection in their boundaries. This saddle connection is a loop through some zero $Z \in Z(\om)$, and these cylinders are given by $C_0^Z,C_1^Z$ for some $Z \in Z(\om)$, possibly after swapping $C_1^Z$ and $C_2^Z$. Since $w_0^Z/w_1^Z \notin \Q$, Lemma \ref{lem:reltwist} tells us that the intersection
\be
W = \Tw_{(X,\om)}\cM \cap {\rm span}_\R (\eta_{C_0^Z},\eta_{C_1^Z},\eta_{C_2^Z})
\ee
has dimension at least $2$. Since $\ker(p) \cap {\rm span}_\R(\eta_{C_0^Z},\eta_{C_1^Z},\eta_{C_2^Z})$ is $1$-dimensional, the dimension of $p(W)$ is at least $1$. Additionally, $\Tw_{(X,\om)}\cM$ contains the standard shear $\eta_\om$, and since $(X,\om)$ has $3g-3 > 3$ horizontal cylinders, $\eta_\om \notin W + \ker(p)$. Thus, the dimension of $p(\Tw_{(X,\om)}\cM)$ is at least $2$, hence $\rank(\cM) \geq 2$.
\end{proof}

In light of Lemma \ref{lem:rank2} and the preceding remark about torus covers, to prove Theorem \ref{thm:sat} it is enough to consider the case of rank at least $2$. In this case, we will be interested in finding flat pairs of pants such that the reciprocals of the associated cylinder circumferences are linearly independent over $\Q$. The following lemma will allow us to find such a flat pair of pants.

\begin{lem} \label{lem:ratiocircum}
Let $\cM$ be an arbitrary orbit closure in a stratum with $\rank(\cM) \geq 2$. Fix a horizontally periodic $(X,\om) \in \cM$ with $\Tw_{(X,\om)}\cM = \Pres_{(X,\om)}\cM$. There exists a nearby horizontally periodic $(X^\pr,\om^\pr) \in \cM$ and $2$ horizontal cylinders on $(X^\pr,\om^\pr)$ whose circumferences have a ratio that is not in $\ol{\Q}$.
\end{lem}

\begin{proof}
The {\em horizontal stretch} $\eta_h \in T^\R_{(X,\om)}\cM$ is defined by
\be
\eta_h(\gam) = \re \int_\gam \om
\ee
for any $\gam \in H_1(X,Z(\om);\Z)$. Since $\Tw_{(X,\om)}\cM = \Pres_{(X,\om)}\cM$, the codimension of $\Pres_{(X,\om)}\cM$ in $T^\R_{(X,\om)}\cM$ is at least $2$, so there exists
\be
\eta \in T^\R_{(X,\om)}\cM \sm \left(\Pres_{(X,\om)}\cM + \R \eta_h\right) .
\ee
For any horizontal cylinder $C$ on $(X,\om)$ and any rightward closed geodesic $\gam_C \subset C$, the value $\eta_h(\gam_C)$ is the circumference of $C$. The horizontal saddle connections and horizontal cylinders on $(X,\om)$ persist on the nearby holomorphic $1$-forms $(X,\om) + t\eta$ (see Section 4 in \cite{Wri:cylinder}), so $(X,\om) + t\eta$ is also horizontally periodic. Let $w_C(t)$ denote the circumference of $C$ on $(X,\om) + t\eta$. Suppose that for any $2$ horizontal cylinders $C,C^\pr$ on $(X,\om)$, the ratio
\be
\frac{w_C(t)}{w_{C^\pr}(t)} = \frac{\eta_h(\gam_C) + t\eta(\gam_C)}{\eta_h(\gam_{C^\pr}) + t\eta(\gam_{C^\pr})}
\ee
is constant. By differentiating with respect to $t$, we see that
\be
\eta_h(\gam_C) \eta(\gam_{C^\pr}) - \eta_h(\gam_{C^\pr}) \eta(\gam_C) = 0 .
\ee
Then there is a nonzero $a \in \R$ such that $\eta(\gam_C) = a\eta_h(\gam_C)$ for all horizontal cylinders $C$ on $(X,\om)$. This means
\be
\eta - a\eta_h \in \Pres_{(X,\om)}\cM ,
\ee
a contradiction. Thus, there exist horizontal cylinders $C,C^\pr$ such that $w_C(t)/w_{C^\pr}(t)$ is not constant. Since $\ol{\Q}$ is countable, there are arbitrarily small values of $t \in \R$ such that on $(X,\om) + t\eta$, the ratio of the circumferences of $C$ and $C^\pr$ is not in $\ol{\Q}$.
\end{proof}

\begin{rmk}
The statement of Lemma \ref{lem:ratiocircum} is stronger than what we will need. We will only need a ratio of circumferences that is not in $\bigcup_{d \in \Z_{>0}}\Q(\sqrt{d})$.
\end{rmk}

The surgery for navigating leaves of the absolute period foliation in the proof of Lemma \ref{lem:pop} can be generalized as follows. See Section 5 in \cite{CDF:transfer} for a similar discussion. Fix $Z \in Z(\om)$, and let $\varphi_1 : [0,T] \ra X$ be an embedded piecewise-geodesic path starting at $Z$. Since the cone angle around $Z$ is $4\pi$, there is at most $1$ other piecewise-geodesic path $\varphi_2 : [0,T] \ra X$ starting at $Z$, such that for all $t \in [0,T]$,
\be
\int_{\varphi_1([0,t])} \om = \int_{\varphi_2([0,t])} \om .
\ee
In other words, $\varphi_1$ and $\varphi_2$ go in the same direction at corresponding times. When it exists, $\varphi_2$ is called the {\em twin path} of $\varphi_1$. The angle between the start of $\varphi_1$ and the start of $\varphi_2$ is $2\pi$. Suppose additionally that $\varphi_2$ is embedded and disjoint from $\varphi_1$ except at $Z$. For $j = 1,2$, slit $X$ along $\varphi_j$ and denote the left and right sides by $\varphi_{j,\ell},\varphi_{j,r}$. Reglue opposite sides of these slits by identifying $\varphi_{j,\ell}(t)$ with $\varphi_{3-j,r}(t)$ for all $t \in [0,T]$. The result is a holomorphic $1$-form ${\rm Sch}_{\varphi_1}(X,\om) \in \Om\cM_g(1^{2g-2})$ in the leaf of the absolute period foliation through $(X,\om)$. By replacing $T$ with $T^\pr$ for $0 \leq T^\pr \leq T$, we obtain a path in this leaf from $(X,\om)$ to ${\rm Sch}_{\varphi_1}(X,\om)$. This surgery can also be carried out when $\varphi_1$ is an embedded piecewise-geodesic loop starting at $Z$, and the twin path $\varphi_2$ is an embedded path disjoint from $\varphi_1$ except at $Z$, and similarly yields a path in the leaf through $(X,\om)$ from $(X,\om)$ to ${\rm Sch}_{\varphi_1}(X,\om) \in \Om\cM_g(1^{2g-2})$. On $(X^\pr,\om^\pr) = {\rm Sch}_{\varphi_1}(X,\om)$, there are {\em inverse paths} $\ol{\varphi_1},\ol{\varphi_2} : [0,T] \ra X^\pr$ starting at $Z$ and parametrized so that for all $t \in [0,T]$,
\be
\int_{\ol{\varphi}_1([0,t])} \om^\pr = \int_{\varphi_1([T-t,T])} \om .
\ee
The paths $\ol{\varphi_1},\ol{\varphi_2}$ are twin paths and are embedded and disjoint, and satisfy
\be
{\rm Sch}_{\ol{\varphi}_1} {\rm Sch}_{\varphi_1} (X,\om) = (X,\om) .
\ee
Moreover, $\varphi_1$ and $\ol{\varphi_1}$ determine the same path in the leaf of the absolute period foliation through $(X,\om)$, going in opposite directions.

\begin{figure}
    \centering
    \includegraphics[width=0.7\textwidth]{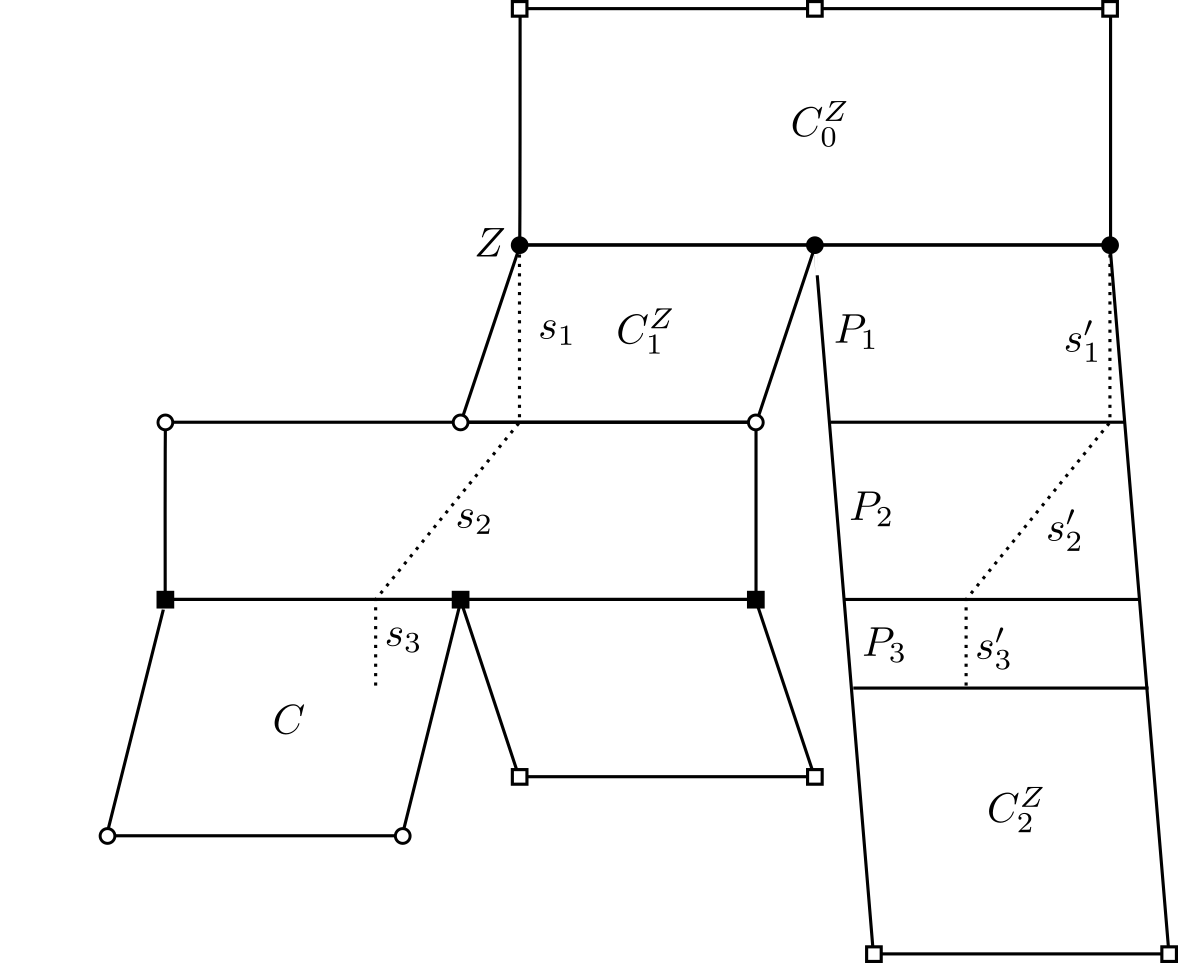}
    \caption{An example of twin paths $\varphi_1 = s_1 \cup s_2 \cup s_3$ and $\varphi_2 = s_1^\pr \cup s_2^\pr \cup s_3^\pr$ on $(X_{2,T},\om_{2,T}) \in \Om\cM_3(1^4)$. The extra horizontal lines in $C_2^Z$ indicate the boundaries of the parallelograms $P_1,P_2,P_3$.}
    \label{fig:sch}
\end{figure}

We will be interested in how this surgery changes the cylinder decomposition of a horizontally periodic holomorphic $1$-form $(X,\om)$. The {\em cylinder digraph} $\Gam(X,\om)$ is the directed graph with a vertex for each horizontal cylinder $C$, and an edge from $C_1$ to $C_2$ for each horizontal saddle connection in the intersection of the top boundary of $C_2$ with the bottom boundary of $C_1$. In general, $\Gam(X,\om)$ is strongly connected, meaning there is a directed path from any vertex to any other vertex.

We are now ready to complete the proof of Theorem \ref{thm:sat} by addressing the case of rank at least $2$.

\begin{thm} \label{thm:rank2}
If $\rank(\cM) \geq 2$, then $\cM$ is the principal stratum.
\end{thm}

\begin{proof}
By Lemma \ref{lem:pop}, there is a horizontally periodic $(X_0,\om_0) \in \cM$ with $3g-3$ horizontal cylinders. Since $(X_0,\om_0)$ has the maximum possible number of horizontal cylinders, $\Tw_{(X_0,\om_0)}\cM = \Pres_{(X_0,\om_0)}\cM$. By Lemma \ref{lem:ratiocircum}, there is a nearby horizontally periodic $(X,\om) \in \cM$ with $3g-3$ horizontal cylinders such that two of the horizontal cylinders have circumferences whose ratio is not in $\ol{\Q}$. It follows that there are two such cylinders on $(X,\om)$ that share a horizontal saddle connection in their boundary. This saddle connection is a loop through a zero $Z \in Z(\om)$, and these two cylinders are given by $C_0^Z,C_1^Z$, possibly after swapping $C_1^Z$ and $C_2^Z$. Since $w_0^Z/w_1^Z \notin \ol{\Q}$, the reciprocals $\frac{1}{w_0^Z},\frac{1}{w_1^Z},\frac{1}{w_2^Z}$ are linearly independent over $\Q$. Then by Lemma \ref{lem:reltwist},
\be
{\rm span}_\R(\eta_{C_0^Z},\eta_{C_1^Z},\eta_{C_2^Z}) \subset \Tw_{(X,\om)}\cM .
\ee

Since $T_{(X,\om)}\cM$ is a complex subspace of $H^1(X,Z(\om);\C)$, we also have $i\eta_{C_j^Z} \in T_{(X,\om)}\cM$ for $j = 0,1,2$. The locally defined path $(X_{j,t},\om_{j,t}) = (X,\om) + it\eta_{C_j^Z} \in \cM$ extends to a path that is well-defined for all $t > -1$, given by vertically scaling $C_j^Z$. On $(X_{j,t},\om_{j,t})$, the height of $C_j^Z$ is given by $h_j^Z(t) = (1+t)h_j^Z$ while the circumference of $C_j^Z$ is unchanged and all other horizontal cylinders on $(X,\om)$ are unchanged. Every horizontal saddle connection and every horizontal cylinder on $(X,\om)$ persists on $(X_{j,t},\om_{j,t})$, so there is an obvious identification of the cylinder digraphs $\Gam(X,\om) \cong \Gam(X_{j,t},\om_{j,t})$. Let $H$ be the sum of the heights of all horizontal cylinders on $(X,\om)$, and choose $T > 0$ such that
\be
\min_j h_j^Z(T) > 2H .
\ee

Choose a horizontal cylinder $C \notin \{C_0^Z,C_1^Z,C_2^Z\}$ on $(X,\om)$, and consider the holomorphic $1$-form $(X_{2,T},\om_{2,T})$. The circumference of $C$ is the same on $(X,\om)$ and on $(X_{2,T},\om_{2,T})$, so let $w$ be this circumference. Choose a directed path $\rho$ in the cylinder digraph from $C_0^Z$ to $C$. We may assume that $\rho$ does not contain a loop, so $\rho$ passes through exactly one of $C_1^Z,C_2^Z$. Possibly after swapping $C_1^Z$ and $C_2^Z$, we may assume $\rho$ passes through $C_1^Z$. Denote the horizontal cylinders in $\rho$ by
\be
C_0,C_1,C_2,\dots,C_n,
\ee
where $C_0 = C_0^Z$, $C_1 = C_1^Z$, and $C_n = C$. On $(X_{2,T},\om_{2,T})$, choose a straight segment $s_1 \subset \ol{C_1}$ starting at $Z$, crossing $C_1$ from top to bottom, and ending at a point in $(\ol{C_1} \cap \ol{C_2}) \sm Z(\om)$. Similarly, for $2 \leq k \leq n-1$, choose a straight segment $s_k \subset \ol{C_k}$ starting at the endpoint of $s_{k-1}$, crossing $C_k$ from top to bottom, and ending at a point in $(\ol{C_k} \cap \ol{C_{k+1}}) \sm Z(\om)$. Lastly, choose a straight segment $s_n \subset \ol{C_n}$ starting at the endpoint of $s_{n-1}$ and ending at a point in $C_n$. The union $s_1 \cup \cdots \cup s_n$ forms a piecewise-geodesic path $\varphi_1$. On $(X_{2,T},\om_{2,T})$, the height of $C_2$ is greater than the sum of the heights of the cylinders in $\rho$. Thus, the twin path $\varphi_2$ is an embedded path in $C_2^Z \cup \{Z\}$ and is disjoint from $\varphi_1$ except at $Z$. We can write $\varphi_2$ as a union of straight segments $s_1^\pr \cup \cdots \cup s_n^\pr$ so that $s_j,s_j^\pr$ are parallel and of the same length.

The holomorphic $1$-form ${\rm Sch}_{\varphi_1}(X_{2,T},\om_{2,T}) \in \cM$ can be viewed as the result of slitting along $s_1,s_1^\pr$ and regluing opposite sides, then slitting along $s_2,s_2^\pr$ and regluing opposite sides, and so on up through $s_n,s_n^\pr$. Let $P_k \subset C_2^Z$ be the open parallelogram bounded by the two sides of $s_k^\pr$ and by the two closed horizontal geodesics through the endpoints of $s_k^\pr$. When we slit along $s_1,s_1^\pr$ and reglue opposite sides, we are equivalently removing $P_1$ from $C_1^Z$ and adding it to $C_1$ by gluing it in along the two sides of $s_1$. As a result, $P_1$ and the cylinder $C_1$ become part of the cylinder $C_0$, and the height of $C_0$ increases by $h_1^Z$. Similarly, for $2 \leq k \leq n-1$, when we slit along $s_k,s_k^\pr$ and reglue opposite sides, we remove $P_k$ from $C_2^Z$ and glue in $P_k$ along the two sides of $s_k$. As a result, the circumference of $C_k$ increases by $w_2^Z$, while the height of $C_2^Z$ decreases by the height of $C_k$. Lastly, when we slit and reglue along $s_n,s_n^\pr$, we remove $P_n$ from $C_2^Z$ and glue in $P_n$ along the two sides of $s_n$. As a result, the height of $C$ decreases, and a new cylinder $C^\pr$ forms with height equal to the height of $s_n$ and circumference $w + w_2^Z$. Thus, on ${\rm Sch}_{\varphi_1}(X_{2,T},\om_{2,T})$, there are exactly $3$ horizontal cylinders $C,C^\pr,C_2^Z$ containing $Z$ in their boundary, and these cylinders form a flat pair of pants. See Figure \ref{fig:sch}.

Suppose that $w/w_2^Z \notin \ol{\Q}$. Then by Lemma \ref{lem:reltwist}, applied to the cylinders $C,C^\pr,C_2^Z$ on ${\rm Sch}_{\varphi_1}(X_{2,T},\om_{2,T})$, we have
\be
{\rm span}_\R (\eta_C,\eta_{C^\pr},\eta_{C_2^Z}) \subset \Tw_{{\rm Sch}_{\varphi_1}(X_{2,T},\om_{2,T})}\cM .
\ee
In particular, for $t \in \R$,
\be
{\rm Sch}_{\varphi_1}(X_{2,T},\om_{2,T}) + t \eta_C \in \cM .
\ee
This means that after slitting and regluing along $\varphi_1$ and its twin path, we can shear $C$ while remaining in $\cM$. Let $\ol{\varphi_1}$ be an inverse path to $\varphi_1$ on ${\rm Sch}_{\varphi_1}(X_{2,T},\om_{2,T})$. This shear can also be described by cutting $C$ along a horizontal closed geodesic $\gam_C$ disjoint from $\varphi_1$, twisting one side relative to the other, and regluing. Since $\varphi_1$ and $\gam_C$ are disjoint, these two surgeries can be applied in either order to get the same holomorphic $1$-form, and similarly for $\ol{\varphi_1}$ and $\gam_C$. In other words, for $t \in \R$, we have
\be
(X_{2,T},\om_{2,T}) + t\eta_C = {\rm Sch}_{\ol{\varphi_1}} \left({\rm Sch}_{\varphi_1}(X_{2,T},\om_{2,T}) + t\eta_C\right) \in \cM ,
\ee
and thus $\eta_C \in \Tw_{(X_{2,T},\om_{2,T})}\cM$. Then by parallel transport along the path $(X_{2,t},\om_{2,t})$, we see that $\eta_C \in \Tw_{(X,\om)}\cM$.

Next, suppose that $w/w_2^Z \in \ol{\Q}$. We may assume that $\Gam(X,\om)$ does not contain a directed path from $C_2^Z$ to $C$ disjoint from $C_0^Z$. Otherwise, since $w/w_1^Z \notin \ol{\Q}$, we could apply the previous argument with $C_1^Z$ and $C_2^Z$ swapped. Choose a directed path $\rho$ in $\Gam(X,\om)$ from $C_1^Z$ to $C$ that does not contain a loop. Additionally, choose an embedded loop $\ell$ in $\Gam(X,\om)$ starting and ending at $C_0^Z$ and passing through $C_2^Z$. Then the path $\rho$ and the loop $\ell$ are disjoint. On $(X_{1,T},\om_{1,T})$, choose a piecewise-geodesic loop $\varphi_1$ starting at $Z$, ending at a point in $\gam_2^Z \sm \{Z\}$, and passing through the saddle connections in $\ell$ in order as before. The cylinders $C_0^Z,C_1^Z$ persist along the path from $(X_{1,T},\om_{1,T})$ to ${\rm Sch}_{\varphi_1}(X_{1,T},\om_{1,T})$ associated to $\varphi_1$. On ${\rm Sch}_{\varphi_1}(X_{1,T},\om_{1,T})$, there is exactly $1$ other horizontal cylinder that contains $Z$ in its boundary, which has circumference $w_1^Z + w_2^Z$, and which we denote $C_2^Z$. The cylinder $C$ also persists along this path. Since $w/w_2^Z \in \ol{\Q}$ and $w_1^Z/w_2^Z \notin \ol{\Q}$, it must be that $w/(w_1^Z + w_2^Z) \notin \ol{\Q}$. The previous argument, with ${\rm Sch}_{\varphi_1}(X_{1,T},\om_{1,T})$ in place of $(X_{2,T},\om_{2,T})$, now implies that $\eta_C \in \Tw_{{\rm Sch}_{\varphi_1}(X_{1,T},\om_{1,T})}\cM$. As before, we then see that $\eta_C \in \Tw_{(X_{1,T},\om_{1,T})}\cM$ and then conclude that $\eta_C \in \Tw_{(X,\om)}\cM$.

We have shown that $\eta_C \in \Tw_{(X,\om)} \cM$ for all horizontal cylinders $C$ on $(X,\om)$. Since the $\eta_C$ are linearly independent over $\R$, and since $(X,\om)$ has $3g-3$ horizontal cylinders, the dimension of $\Tw_{(X,\om)}\cM$ is at least $3g-3$. Since $\om$ has $2g-2$ distinct zeros, the real dimension of $\ker(p) \cap \Tw_{(X,\om)}\cM$ is at most $2g-3$. Then
\be
\rank(\cM) \geq \dim_\R p(\Tw_{(X,\om)} \cM) \geq (3g-3) - (2g-3) = g ,
\ee
so $\cM$ has rank $g$. Then since $\cM$ is saturated for the absolute period foliation, we have
\be
\dim_\C \cM = 2g + (2g-3) = 4g-3 = \dim_\C \Om\cM_g(1^{2g-2}) ,
\ee
hence $\cM = \Om\cM_g(1^{2g-2})$.
\end{proof}

\begin{rmk} \label{rmk:fpp}
The argument in the proof of Theorem \ref{thm:rank2} applies more broadly to orbit closures in any stratum $\Om\cM_g(k_1,\dots,k_n)$ with some $k_j = 1$. Specifically, suppose $(X,\om)$ is a horizontally periodic holomorphic $1$-form in an orbit closure $\cM \subset \Om\cM_g(k_1,\dots,k_n)$ with $3$ horizontal cylinders $C_0,C_1,C_2$ forming a flat pair of pants. If $\eta_{C_0},\eta_{C_1},\eta_{C_2} \in T_{(X,\om)}\cM$, and if the circumferences $w_0,w_1,w_2$ satisfy $w_1/w_2 \notin \bigcup_{d \in \Z_{>0}}\Q(\sqrt{d})$, then $\eta_C \in T_{(X,\om)}\cM$ for all horizontal cylinders $C$ on $(X,\om)$.
\end{rmk}

Lastly, we turn to the proof of Theorem \ref{thm:free}, which strengthens the arguments in the proof of Theorem \ref{thm:rank2} using an analysis of the space of directed loops in the cylinder digraph. Consider the stratum $\Om\cM_g(k_1,\dots,k_n)$ of holomorphic $1$-forms with exactly $n$ distinct zeros with orders $k_1,\dots,k_n$. Let $\cM \subset \Om\cM_g(k_1,\dots,k_n)$ be an orbit closure. A {\em free pair of pants} on a holomorphic $1$-form in $\cM$ is a collection of $3$ parallel cylinders $C_0^Z,C_1^Z,C_2^Z$ forming a flat pair of pants as in Figure \ref{fig:fpp}, such that each $C_j^Z$ can be sheared while remaining in $\cM$. When $(X,\om)$ is horizontally periodic, this means $\eta_{C_0^Z},\eta_{C_1^Z},\eta_{C_2^Z} \in T_{(X,\om)}\cM$.

Fix a horizontally periodic $(X,\om) \in \Om\cM_g(k_1,\dots,k_n)$. Let $\Gam = \Gam(X,\om)$ be the cylinder digraph, and let $V(\Gam)$ and $E(\Gam)$ be the associated sets of vertices and edges. Since $|E(\Gam)|$ is the number of horizontal saddle connections on $(X,\om)$, we have
\be
|E(\Gam)| = \sum_{j=1}^n (k_j + 1) = 2g - 2 + n .
\ee
Let $\Q^{E(\Gam)}$ be the $\Q$-vector space of formal $\Q$-linear combinations of directed edges in $\Gam$. Let $L(\Gam) \subset \Q^{E(\Gam)}$ be the $\Q$-vector subspace spanned by the directed loops in $\Gam$. Note that $L(\Gam)$ is also spanned by the embedded directed loops in $\Gam$, which pass through each vertex at most once. Since $\Gam$ is strongly connected, by Lemma 5.1 in \cite{MW:fullrank} we have
\be
\dim_\Q L(\Gam) = |E(\Gam)| - |V(\Gam)| + 1 = \dim_\C \Om\cM_g(k_1,\dots,k_n) - |V(\Gam)| .
\ee

For each horizontal cylinder $C$ on $(X,\om)$, choose a saddle connection $\gam_C \subset C \cup Z(\om)$ crossing $C$ from bottom to top. Let $\ell$ be an embedded directed loop in $\Gam$, and let $C_1,\dots,C_m$ and $\gam_1,\dots,\gam_m$ be the associated cyclically ordered sequences of horizontal cylinders and horizontal saddle connections. The saddle connections $\gam_j$ are oriented rightward, and the cylinders $C_j$ are indexed so that $\gam_j \cup \gam_{j+1} \subset \ol{C_j}$, where indices are taken modulo $m$. Let $\wt{\ell} \subset X \sm Z(\om)$ be an embedded piecewise-geodesic loop given by a union of straight segments $s_1 \cup \cdots \cup s_m$, such that $s_j \subset (C_j \cup \gam_j \cup \gam_{j+1}) \sm Z(\om)$ crosses $C_j$ from top to bottom and is disjoint from $\gam_{C_j}$. The loop $\wt{\ell}$ is determined up to free homotopy in $X \sm Z(\om)$. For $t > 0$, define $(X_t,\om_t)$ by cutting $(X,\om)$ along $\wt{\ell}$ and gluing in parallelograms $P_{j,t}$, such that one pair of sides of $P_{j,t}$ is horizontal and of length $t$, and the other pair of sides arise from cutting along $s_j$. The path $(X_t,\om_t)$ extends to a path in $\Om\cM_g(k_1,\dots,k_n)$ that is well-defined for $t > -\min_j \int_{\gam_j} \om$. The associated tangent vector $\eta_\ell \in H^1(X,Z(\om);\R)$ is the relative cohomology class that evaluates to $1$ on each of the horizontal saddle connections $\gam_j$ and evaluates to $0$ on all other horizontal saddle connections and on the saddle connections $\gam_{C_j}$. If $\ell_1,\dots,\ell_M$ is a $\Q$-basis for $L(\Gam)$, then
\be
\{\eta_{\ell_1},\dots,\eta_{\ell_M}\} \cup \{\eta_C \; : \; C \in V(\Gam)\}
\ee
is an $\R$-basis for $H^1(X,Z(\om);\R)$. In light of Remark \ref{rmk:fpp}, to prove Theorem \ref{thm:free} it is enough to show that all of the $\eta_{\ell_j}$ lie in the tangent space to $\cM$.

\begin{figure}
    \centering
    \includegraphics[width=\textwidth]{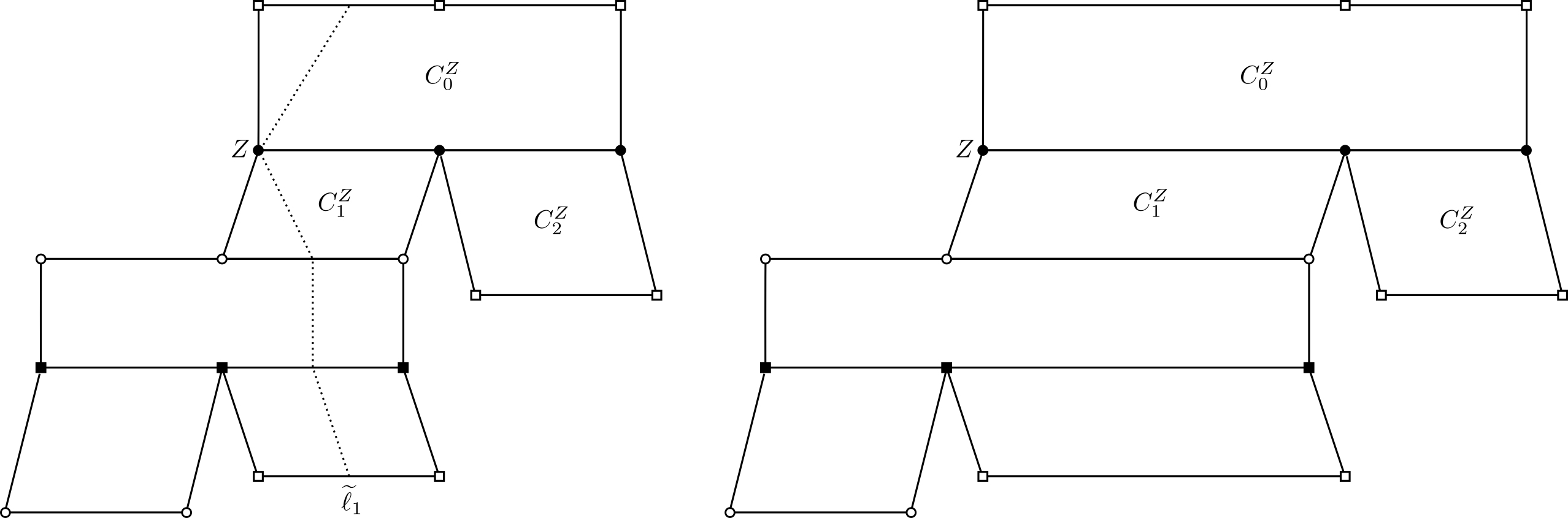}
    \caption{Left: $(X,\om)$ with the loop $\wt{\ell}_1$ indicated by dashes. Right: the result of stretching $C_2^Z$, slitting $\wt{\ell}_1$ and its twin path and regluing, and then restoring $C_2^Z$ to its original height.}
    \label{fig:sch2}
\end{figure}

\begin{proof} (of Theorem \ref{thm:free})
By rotating $(X,\om) \in \cM$, we may assume that there are horizontal cylinders on $(X,\om)$ forming a free pair of pants. By Corollary 6 in \cite{SW:minimal}, the horocycle through $(X,\om)$ accumulates on a horizontally periodic holomorphic $1$-form. Since the horocycle flow preserves lengths of horizontal saddle connections and the heights and circumferences of horizontal cylinders, we may additionally assume that $(X,\om)$ is horizontally periodic. Let $C_0^Z,C_1^Z,C_2^Z$ be horizontal cylinders in a free pair of pants on $(X,\om)$ as in Figure \ref{fig:fpp}. Since $T_{(X,\om)}\cM$ is a complex subspace of $H^1(X,Z(\om);\C)$, we have
\be
{\rm span}_\C (\eta_{C_0^Z},\eta_{C_1^Z},\eta_{C_2^Z}) \subset T_{(X,\om)}\cM .
\ee
For each horizontal cylinder $C$ on $(X,\om)$, choose a saddle connection $\gam_C \subset C \cup Z(\om)$. Then each embedded loop $\ell \subset \Gam$ determines an embedded piecewise-geodesic loop $\wt{\ell} \subset X \sm Z(\om)$ disjoint from the saddle connections $\gam_C$, up to free homotopy in $X \sm Z(\om)$ as above.

First, suppose that $\ell$ passes through $C_0^Z$. Since $\ell$ is embedded, $\ell$ passes through exactly one of $C_1^Z$ or $C_2^Z$. We may assume that $\ell$ passes through $C_1^Z$, possibly after swapping $C_1^Z$ and $C_2^Z$. Let $s_1,\dots,s_m$ be the segments in $\wt{\ell}$ in cyclic order, indexed so that $s_1$ crosses $C_1^Z$ from top to bottom. Then $s_m$ crosses $C_0^Z$ from top to bottom. Let $\wt{s}_1$ be the straight segment in $\ol{C_1^Z}$, from $Z$ to the endpoint of $s_1$, that is disjoint from $\gam_{C_1^Z}$ except at $Z$. Let $\wt{s}_m$ be the straight segment in $\ol{C_0^Z}$, from the endpoint of $s_{m-1}$ to $Z$, that is disjoint from $\gam_{C_0^Z}$ except at $Z$. Let $\wt{\ell}_1 = \wt{s}_1 \cup s_2 \cup \cdots \cup s_{m-1} \cup \wt{s}_m$. Vertically stretch $C_2^Z$ so that its height is greater than the sum of the heights of the other horizontal cylinders on $(X,\om)$, and let $(X_1,\om_1)$ be the resulting holomorphic $1$-form. Let $\varphi_1 = \wt{\ell_1}$, treated as a parametrized loop in $(X_1,\om_1)$ starting at $Z$. The twin path $\varphi_2$ is an embedded path contained in $C_2^Z \cup \{Z\}$. Slit along $\varphi_1 \cup \varphi_2$ and reglue opposite sides to get a holomorphic $1$-form ${\rm Sch}_{\varphi_1}(X_1,\om_1)$, and then vertically shrink $C_2^Z$ so that its height is equal to its height on $(X,\om)$ to get a holomorphic $1$-form $(X_2,\om_2)$. The free pair of pants on $(X,\om) \in \cM$ and the fact that $\cM$ is locally defined by homogeneous linear equations ensures that $(X_2,\om_2) \in \cM$. Equivalently, $(X_2,\om_2)$ is obtained by slitting $(X,\om)$ along $\wt{\ell}$ and gluing in parallelograms of width $w_2^Z$ along the sides of each segment in $\wt{\ell}$. We have shown that
\be
(X_2,\om_2) = (X,\om) + w_2^Z \eta_{\ell} \in \cM .
\ee
See Figure \ref{fig:sch2}.

Next, suppose that $\ell$ does not pass through $C_0^Z$. Then $\ell$ does not pass through $C_1^Z$ or $C_2^Z$. Possibly after swapping $C_1^Z$ and $C_2^Z$, there is an embedded path $\rho$ in $\Gam$ from $C_1^Z$ to a vertex in $\ell$ that is otherwise disjoint from $\ell$. Let $D_1,\dots,D_p$ and $\del_2,\dots,\del_p$ be the associated sequences of horizontal cylinders and horizontal saddle connections, indexed so that $D_1 = C_1^Z$ and so that $\del_j \cup \del_{j+1} \subset \ol{D_j}$ for $j = 2,\dots,p-1$. Let $r_1$ be the straight segment in $\ol{C_1^Z}$, from $Z$ to a point in $\del_2 \sm Z(\om)$, that is disjoint from $\gam_{C_1^Z}$ except at $Z$. Let $r_p$ be a straight segment in $\ol{D_p} \sm Z(\om)$, from a point in $\del_p \sm Z(\om)$ to the point on $\wt{\ell}$ in the bottom boundary of $D_p$, that is disjoint from $\gam_{D_p}$ and is disjoint from $\wt{\ell}$ except at its endpoint. Let $\wt{\rho} = r_1 \cup r_2 \cup \cdots \cup r_{p-1} \cup r_p$ be an embedded piecewise-geodesic path such that for $j = 2,\dots,p-1$, $r_j \subset (D_j \cup \del_j \cup \del_{j+1}) \sm Z(\om)$ is a straight segment crossing $D_j$ from top to bottom and disjoint from $\gam_{D_j}$. Vertically stretch $C_2^Z$ so that its height is greater than twice the sum of the heights of the other horizontal cylinders on $(X,\om)$, and let $(X_1,\om_1)$ be the resulting holomorphic $1$-form. Let $\varphi_1 = \wt{\rho}_1$, treated as a parametrized path in $(X_1,\om_1)$ starting at $Z$. The twin path $\varphi_2$ is embedded and contained in $C_2^Z \cup \{Z\}$. Slit along $\varphi_1 \cup \varphi_2$ and reglue opposite sides to get a holomorphic $1$-form ${\rm Sch}_{\varphi_1}(X_1,\om_1) \in \cM$, and let $\ol{\varphi}_1$, $\ol{\varphi}_2$ be the inverse twin paths on ${\rm Sch}_{\varphi_1}(X_1,\om_1)$. The loop $\wt{\ell}$ is preserved along the associated path from $(X_1,\om_1)$ to ${\rm Sch}_{\varphi_1}(X_1,\om_1)$, and the twin path of $\wt{\ell}$ on ${\rm Sch}_{\varphi_1}(X_1,\om_1)$ is an embedded path contained in $C_2^Z \cup \{Z\}$. Moreover, $\wt{\ell}$ and its twin are disjoint from the inverse paths $\ol{\varphi}_1 \cup \ol{\varphi}_2$. Thus, we can slit along $\wt{\ell}$ and its twin and reglue opposite sides, then slit along $\ol{\varphi}_1 \cup \ol{\varphi}_2$ and reglue opposite sides, and then restore $C_2^Z$ to its original height, to get a holomorphic $1$-form $(X_2,\om_2)$. Similarly to the previous case, we have
\be
(X_2,\om_2) = (X,\om) + w_2^Z \eta_\ell \in \cM .
\ee

Let $M = \dim_\Q L(\Gam)$, and let $\ell_1,\dots,\ell_M$ be a $\Q$-basis for $L(\Gam)$. There is a map
\be
F : (\R_{>0})^M \ra \Om\cM_g(k_1,\dots,k_n)
\ee
whose image contains $(X,\om)$ and consists of horizontally periodic holomorphic $1$-forms. On $F(t_1,\dots,t_M)$, the length of a horizontal saddle connection $\gam$ is given by $\sum_{j=1}^M t_j \eps_j$, where $\eps_j$ is $1$ or $0$ according to whether or not $\ell_j$ passes through $\gam$. The saddle connections $\gam_C$ on $(X,\om)$ persist on every holomorphic $1$-form in the image of $F$, and we require that the integral along $\gam_C$ is the same for all holomorphic $1$-forms in the image of $F$. Write $(X,\om) = F(a_1,\dots,a_M)$. We have shown that for $k = 1,\dots,M$, there is $j_k \in \{1,2\}$ such that
\be
F(a_1,\dots,a_{k-1},a_k + w_{j_k}^Z,a_{k+1},\dots,a_M) \in \cM .
\ee
By iterating the surgeries in the previous two paragraphs, and by scaling in the horizontal direction by positive real numbers, we then get that
\be
F(t(a_1+n_1 w_{j_1}^Z),\dots,t(a_M+n_M w_{j_M}^Z)) \in \cM
\ee
for all $t > 0$ and all positive integers $n_1,\dots,n_M$. Define
\be
S = \{(t(a_1+n_1 w_{j_1}^Z),\dots,t(a_M+n_M w_{j_M}^Z)) \; : \; t > 0, n_j \in \Z_{>0}\} \subset (\R_{>0})^M .
\ee
Since $(a_1 + n n_k w_{j_k}^Z)/n \ra n_k w_{j_k}^Z$ as $n \ra \infty$, the closure of $S$ contains every tuple of the form $(tn_1w_{j_1}^Z,\dots,tn_Mw_{j_M}^Z)$ with $t > 0$ and $n_j \in \Z_{>0}$. By rescaling the $k$-th coordinate by $1/w_{j_k}^Z$, we see that $S$ is dense if and only if
\be
S_1 = \{(t n_1, \dots, t n_M) \; : \; t > 0, n_j \in \Z_{>0}\}
\ee
is dense. Since $S_1$ contains $(\Q_{>0})^M$, we conclude that $S$ is dense in $(\R_{>0})^M$. Then since $F$ is continuous and $\cM$ is closed, $\cM$ contains the image of $F$. In particular, $\cM$ contains the linear paths $F(a_1,\dots,a_{k-1},a_k+tw_{j_1}^Z,a_{k+1},\dots,a_M)$, thus
\be
\eta_{\ell_k} \in T_{(X,\om)}^\R \cM
\ee
for $k = 1,\dots,M$. Choose $k$ such that $\ell_k$ passes through $C_1^Z$. Since $\ell_k$ is embedded, $\ell_k$ does not pass through $C_2^Z$. Along the locally defined path $(X,\om) + t\eta_{\ell_k} \in \cM$, the circumference of $C_1^Z$ changes while the circumference of $C_2^Z$ is constant. The free pair of pants $C_0^Z,C_1^Z,C_2^Z$ persists along this path, so by replacing $(X,\om)$ with $(X,\om) + t\eta_{\ell_k}$ for some small $t \in \R$, we can arrange that $w_1^Z/w_2^Z \notin \ol{\Q}$. Then by Remark \ref{rmk:fpp}, the proof of Theorem \ref{thm:rank2} shows that $\eta_C \in T_{(X,\om)}^\R\cM$ for all horizontal cylinders $C$ on $(X,\om)$. Since the $\eta_{\ell_j}$ and the $\eta_C$ span $H^1(X,Z(\om);\C)$, we have shown that $\cM$ has the same complex dimension as the ambient stratum $\Om\cM_g(k_1,\dots,k_n)$. Hence, $\cM$ is a connected component of a stratum.
\end{proof}

\begin{figure}
    \centering
    \includegraphics[width=0.6\textwidth]{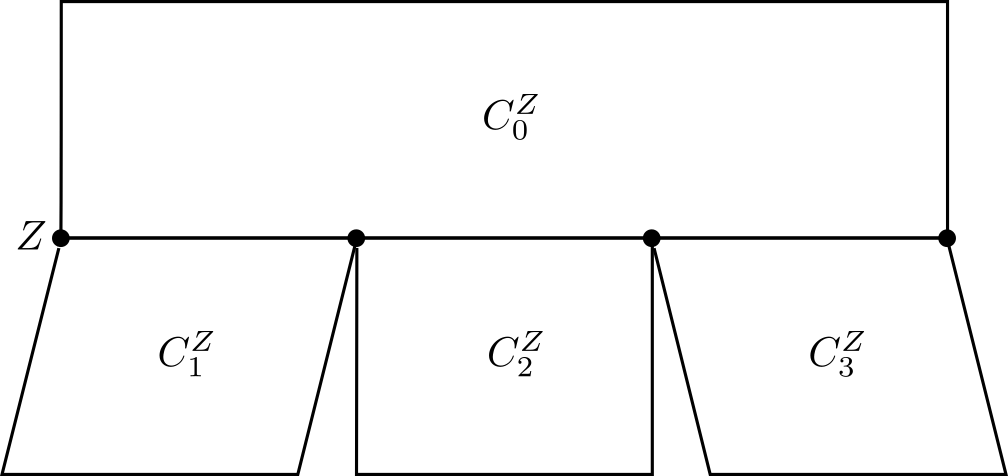}
    \caption{A $k$-legged pants with $k = 3$.}
    \label{fig:fpp2}
\end{figure}

\begin{rmk} \label{rmk:fpp2}
We conclude by remarking that an almost identical argument to the proof of Theorem \ref{thm:free} can be applied to the following configuration of $k+1$ horizontal cylinders $C_0^Z,\dots,C_k^Z$ for $k \geq 2$. Here, $Z$ is a zero of order $k-1$. For $1 \leq j \leq k$, the intersection of the bottom boundary of $C_0^Z$ with the top boundary of $C_j^Z$ is a single horizontal saddle connection $\gam_j^Z$ that is a loop through $Z$. This configuration looks like a ``$k$-legged pants''. See Figure \ref{fig:fpp2}. As in Theorem \ref{thm:free}, if each $C_j^Z$ can be sheared while remaining in a given orbit closure, then that orbit closure must be a stratum component.
\end{rmk}


\bibliographystyle{math}
\bibliography{my.bib}

{\small
\noindent
Email: kwinsor@math.harvard.edu

\noindent
Fields Institute for Research in Mathematical Sciences, Toronto, Canada
}

\end{document}